\documentclass[12pt]{article}

\usepackage[T1]{fontenc}
\usepackage{avant}
\usepackage[cp1250]{inputenc}
\usepackage{amsmath,amssymb,amsthm}

\title{Generalized symplectic symmetric spaces}
\author{Maciej Boche\'nski and Aleksy Tralle}

\begin{document}

\newtheorem{theorem}{Theorem}
\newtheorem{proposition}{Proposition}
\newtheorem{lemma}{Lemma}
\newtheorem{definition}{Definition}
\newtheorem{example}{Example}
\newtheorem{note}{Note}
\newtheorem{corollary}{Corollary}
\newtheorem{rem}{Remark}

\maketitle{}

\begin{abstract}
Bieliavsky introduced and investigated a class of symplectic symmetric spaces, that is, symmetric spaces endowed with a symplectic structure invariant with respect to symmetries. 
Since the theory of symmetric spaces has essential and interesting generalizations, we ask a question about their possible symplectic versions. In this paper we do obtain such generalization, and, in particular, give a list of all symplectic $3$-symmetric manifolds with simple groups of transvections. We also show a method of constructing semisimple (noncompact) symplectic $k$-symmetric spaces from a given (compact) K\"ahler k-symmetric space. 
\end{abstract}

\section{Introduction}

A homogeneous $k$-symmetric space is a homogeneous space generated by a Lie group automorphism of order $k$.  It is a natural generalization of the concept of a symmetric space. The theory of generalized symmetric spaces was worked out by  Fedenko and Kowalski \cite{f}, \cite{kow}. In \cite{b1} Bieliavsky built a symplectic version of the theory of symmetric spaces. In this paper he expressed the classification problem of symplectic symmetric spaces in terms of Lie algebras. In particular, a classification of semisimple symplectic symmetric spaces was obtained by studying infinitesimal objects called symplectic triples. Let $X=(G/K, \omega , s)$ be a simply connected, symmetric symplectic space. One can uniquely identify this manifold with the appropriate triple (called the symplectic triple): $(\mathfrak{g}, \sigma, \Omega ).$ The Lie algebra $\mathfrak{g}$ of a Lie group $G$ has a reductive decomposition into $+1$ and $(-1)$ eigenspaces of an involutive  automorphism $\sigma$ of $\mathfrak{g}$:
$$\mathfrak{g}=\mathfrak{h}+\mathfrak{m}$$
It is shown, that the existence of an invariant 2-form $\Omega$, whose restriction to $\mathfrak{m} \times \mathfrak{m}$ is a symplectic structure, is equivalent to the existence of an element $Z$ in the center $Z(\mathfrak{h}),$ which has the injective adjoint representation on $\mathfrak{m}.$ Then it is proved, that in the symmetric semisimple case, every nonzero element in $Z(\mathfrak{h})$ has this property. This reduces the problem of the classification of simple symplectic spaces to finding appropriate Lie algebras  in  Berger's  list of simple involutive Lie algebras \cite{be}.
Our aim is to extend this result to the class of  $k$-symmetric case. This is not straightforward for several reasons (for example, Berger's list is not applicable). However, using general techniques from \cite{kow} together with modifications of Bieliavsky's arguments and a quite subtle kind of "duality" between $3$-symmetric spaces of compact and non-compact type \cite{g1}, we are able to extend some of his results to the case of generalized symmetric spaces.  In particular, we get a classification of $3$-symmetric spaces with simple transvection groups. The latter is given in Tables 1-3 at the end of the article. Our work consists of three steps: 
\begin{enumerate}

\item  Description of symplectic $k$-symmetric spaces and the equivalence between $k$-symmetric symplectic spaces and $k$-symmetric symplectic triples. This part mainly follows the line of arguments in  \cite{kow} and \cite{b1},
\item Establishing the  properties of $k$-symmetric triples and a connection between symplectic forms and elements of the center of $\mathfrak{h},$ 
\item A classification theorem.
\end{enumerate}
Methods used in this paper are based on classical Lie group theory. If necessary the reader may consult \cite{he} and \cite{ov}. Our notation and terminology are closed to these sources, therefore, we use this material without further explanations. Our basic reference for symplectic manifolds is \cite{mcds}

The main results are contained in Theorems \ref{thm:triples},\ref{thm:sympl-descr}, \ref{thm:classification} and Tables 1-2.

\section{$k$-symmetric spaces}

\begin{definition}
{\rm A {\it regular $k$-symmetric symplectic manifold} is a triple $(M,\omega,s)$, where $M$ is a manifold with a symplectic form $\omega$, and the family of diffeomorphisms
$$s=\{s_{x}:M\rightarrow M, x\in M\}$$ 
satisfies for all $x,y \in M$ the following conditions:}
\begin{enumerate}
	\item $s_{x}(x)=x$,
	\item $s_{x}$ {\rm is a symplectomorphism, such that $s_{x}^{k}=id$ and $k\geq2$ is the least integer with this property,
	\item $s_{x}\circ s_{y}=s_{z}\circ s_{x}$ for $z=s_{x}(y)$,
	\item the fixed point $x$ of $s_{x}$ is isolated.}
\end{enumerate}
\label{dd2}
\end{definition}

We  begin with a general description of  $k$-symmetric spaces following \cite{kow}. 

\begin{definition}
{\rm Let there be given a regular $k$-symmetric manifold $(M,s)$.{\sl The group of transvections $G(M)$} of $(M,s)$ is the group generated by all automorphisms of the form:} $s_{x}\circ s_{y}^{-1}, \ x,y\in M$. 
\end{definition}
\begin{definition}
{\rm Let there be given a regular $k$-symmetric manifold $(M,s)$. Consider the smooth multiplication $M\times M\rightarrow M$ defined by the formula $x\cdot y=s_x(y)$. The group $Aut(M)$ of diffeomorphisms preserving the multiplication will be called {\it the group of automorphisms of $(M,s)$.}}
\end{definition}
It is proved in \cite{kow} that $Aut(M)$ is a Lie group of transformations of $M$. We will need the Proposition below, which is proved in \cite{kow} (Theorems II.33 and II.35).

\begin{proposition}
Let $(M,\omega ,s)$ be a regular $k$-symmetric space. Then $G(M)$ is a connected normal subgroup of $Aut(M)$. Moreover $G(M)$ is a transitive Lie group of transformations of  $M$.
\end{proposition}

\noindent
This Proposition enables us to  look at  $(M,s)$ as at a homogeneous space.

\begin{definition}
{\rm {\it A regular homogeneous $k$-manifold} is a triple $(G,H,\sigma)$, where $G$ is a connected Lie group, $H \subseteq G$ is a closed subgroup, and $\sigma :G \rightarrow G$ is an automorphism of G such that:}
\begin{enumerate}
	\item {\rm $\sigma^{k} =id$ and $k\geq2$ is the least integer with this property,
	\item ${(G^{\sigma})}^{o} \subset H \subset G^{\sigma}$ , where $G^{\sigma} =\{ g\in G \mid \sigma (g)=g \}$ , and ${(G^{\sigma})}^{o}$ is the identity component of $G^{\sigma}$.}
\end{enumerate}
\label{dd}
\end{definition}

\begin{proposition}[II.24 and II.45 in \cite{kow}]
For $(G,H,\sigma)$ let $A:=id- \sigma^{\ast},$  $\mathfrak{h}:=ker \ A ,$ $\mathfrak{m}:=Im \ A.$ Then $\mathfrak{h}$ is a Lie algebra of $H$ and the space $G/H$ is reductive with respect to the decomposition $\mathfrak{g}=\mathfrak{h}+\mathfrak{m}.$ 
\label{ff}
\end{proposition}

\begin{definition}
{\rm $(G,H,\sigma)$ is said to be prime if $G$ acts effectively on $G/H$ and $[\mathfrak{m} ,\mathfrak{m}]_{\mathfrak{h}} =\mathfrak{h}$ (the index at the bracket indicates taking the $\mathfrak{h}$-component of a vector in $\mathfrak{g}$)}.
\end{definition}

\begin{theorem}[II.40 in \cite{kow}]
Let $o\in M$ and $k\geq 2.$ There is a one-to-one correspondence between the pointed, connected, regular $k$-manifolds $(M,s,o)$ and the prime, regular homogeneous $k$-manifolds $(G,H,\sigma).$
\label{t240}
\end{theorem}

\noindent
Before proceeding to the symplectic case, we need one more observation.

\begin{lemma}[Fitting decomposition]
Let $W$ be a finite dimensional vector space and $A:W \rightarrow W$ be an automorphism. There is a unique decomposition $W=W_{0A}\oplus W_{1A}$ of $W$ into $A$-invariant subspaces, such that $A\mid_{W_{0A}}$ is nilpotent, and $A\mid_{W_{1A}}$ is an automorphism. We also have:
$$W_{1A}=\cap_{i=1}^{\infty}A^{i}(W) \ \ W_{0A}= \{ v \| A^{i}(v)=0 \ for \ some \ i \}. $$
\end{lemma}

Let $\mathfrak{g}$ be a Lie algebra with an automorphism $\nu$ of order $k$ and let $A:=id-\nu$. Take $\mathfrak{h}:=ker \ A$ and $\mathfrak{m}:=Im \ A,$ then $\mathfrak{g}=\mathfrak{h}+\mathfrak{m}$ is the Fitting decomposition of $\mathfrak{g}.$ It is sufficient to show, that $\mathfrak{h}=\mathfrak{g}_{0A},$ because $A(h+m)=A(m).$ Assume that $X\in \mathfrak{g}_{0A}$ and $X\notin \mathfrak{h}.$ Without loss of generality we get: $AX \neq 0$ and $A^{2}X=0.$ Thus, $X-\nu X=Y$ for some $Y\in \mathfrak{g}, \ \nu Y=Y,$, and, therefore:
$$\nu X=X-Y$$
$$\nu^{2} X=\nu X - \nu Y=X-2Y$$
$$...$$
$$\nu^{k} X=X-kY=X,$$
and hence $Y=0$,  a contradiction.
We also get 
$$[\mathfrak{h} ,\mathfrak{h}] \subseteq \mathfrak{h} \ and \ [\mathfrak{h} ,\mathfrak{m}] \subseteq \mathfrak{m} \eqno (*)$$
The first one is obvious. As for the second, one can notice that   $A \mid_{\mathfrak{m}}$ is an automorphism. For $X \in \mathfrak{m} , \ Y \in \mathfrak{h}$ there is $Z \in \mathfrak{m}$ such that $X=AZ=Z-\nu Z$. Finally
$$[X,Y]=[Z-\nu Z,Y]=[Z,Y]-[\nu Z, \nu Y]=A[Z,Y]\in \mathfrak{m}.$$

\noindent
Consequently, the pair $(\mathfrak{g}, \nu)$ gives rise to the simply connected regular homogeneous $k$-manifold $(G,H,\sigma)$, where  $\nu = \sigma_{\ast}.$

\begin{definition} {\rm A homogeneous space $G/H$ is called {\it reductive}, if there is a decomposition $\mathfrak{g}=\mathfrak{h}\oplus\mathfrak{m}$ such that $Ad_G(H)(\mathfrak{m})\subset\mathfrak{m}$.} 
\end{definition}
Note that if $H$ is connected, the reductivity condition is equivalent to $(*)$. It is clear from our discussion that regular homogeneous $k$-symmetric spaces are reductive. The property of reductivity will be used throughout the paper. More details can be found in \cite{kow}. 

\begin{definition}
{\rm The pair $(\mathfrak{g}, \nu)$ will be called {\sl $k$-symmetric}. It is effective, if $\mathfrak{h}$ does not contain any proper ideal of $\mathfrak{g}$ and prime, if it is effective and satisfies the equalty:
$$[\mathfrak{m} ,\mathfrak{m}]_{\mathfrak{h}} =\mathfrak{h}.$$
The decomposition $\mathfrak{g}=\mathfrak{h}+\mathfrak{m}$ is called {\it the canonical decomposition of a $k$-symmetric pair}}.
\end{definition}

\noindent
Let there be given a homogeneous $k$-symmetric manifold  $(G,H,\sigma)$. If the group $G$ acts effectively on $G/H,$ then the corresponding pair $(\mathfrak{g} ,\nu)$ is called effective.
Conversely, assume that $(\mathfrak{g} ,\nu)$ is effective.  Then the map $U \rightarrow ad(U) \mid_{\mathfrak{m}}$ defined on $\mathfrak{h}$ is injective. We have the following. Let $\mathfrak{n}$ be the set of all elements $U \in \mathfrak{h}$ such that $[U, \mathfrak{m} ]=0$. For $U^{'} \in \mathfrak{h}, \ U \in \mathfrak{n} , \ X \in \mathfrak{m}$ we have $[U,X]=0 \ and \ [X,U^{'}] \in \mathfrak{m}$ which implies
$$\left[ [U^{'},U],X \right] = \left[U^{'},[U,X] \right] + \left[ U,[X,U^{'}] \right]=0.$$
For $Y \in \mathfrak{h} , \ U \in \mathfrak{n} , \ X \in \mathfrak{m}$ we obtain $\left[ [Y,U],X \right] =0$. Hence $\mathfrak{n}$ is an ideal of $\mathfrak{g}$ and $\mathfrak{n} \subset \mathfrak{h}$, therefore $\mathfrak{n} =(0)$.

\begin{proposition}
Let $k\geq 2.$ There is a one-to-one correspondence between $k$-symmetric, prime pairs $(\mathfrak{g} , \nu )$ and simply connected prime homogeneous $k$-manifolds  $(G,H,\sigma)$. The group $G$ corresponds to $\mathfrak{g}$, $\nu = \sigma_{\ast}$ and $H$ is a connected subgroup with the Lie algebra $\mathfrak{h}=kerA$, where $A=id-\nu$.
\label{ff24}
\end{proposition}

\noindent
We also need the following.

\begin{proposition}
Let $(\mathfrak{g} ,\nu)$ be a $k$-symmetric pair in which $\mathfrak{g}$ is a simple Lie algebra. Then $(\mathfrak{g} ,\nu)$ is prime.
\label{psp}
\end{proposition}

\begin{proof}
It suffices to notice, that $\mathfrak{c}:=\mathfrak{m}+[\mathfrak{m},\mathfrak{m}]$ is an ideal of $\mathfrak{g}.$
\end{proof}

\begin{definition}
{\rm A connected regular pointed $k$-symmetric space $(M,s,o)$ is called {\it simple}, if $G=G(M)$ is a simple Lie group.}
\end{definition}

Now it is easy to see that we can identify isomorphic classes of simply connected simple $k$-symmetric spaces $(M,s,o)$ with the isomorphic classes of simply connected prime homogeneous $k$-manifolds $(G,H,\sigma),$ and isomorphic classes of $k$-symmetric prime pairs $(\mathfrak{g} , \nu ).$ We can use interchangeably those objects as a description of a fixed $k$-symmetric manifold.
Now we must check, if the above construction works in the symplectic category. We need the following theorem:

\begin{theorem}[II.25 in \cite{kow}]
.
\begin{enumerate}
 \item Let $(G,H, \sigma )$ be a regular homogeneous $k$-manifold, $\pi :G \rightarrow G/H$ the canonical projection and $s$ the transformation of $G/H$ induced by $\sigma$, $s \circ \pi = \pi \circ \sigma$ on $G$.
 Let $G$ act on $G/H$ by left translations. Then putting:
   \begin{equation}
   s_{\pi (g)} = g\circ s\circ g^{-1} \ for \ each \ g\in G,
   \label{eq1}
   \end{equation}
 we obtain a well defined family $p=\{ s_{x} \mid x\in G/H \}$ of diffeomorphisms of $G/H$. $(G/H,p)$ is a connected, regular k-manifold and $G$ acts as a group of automorphisms of it.
 \item Conversely, let $(M,p,o)$ be a connected, regular $k$-manifold and $o\in M$ be a fixed point. Let $G$ be the identity component of the automorphism group $Aut(M)$ and $G_{o}$ the isotropy subgroup of $G$ at $o$. Define a map $\sigma :G \rightarrow Aut(M)$ by the formula
 $$\sigma (g)=s_{o} \circ g \circ s_{o}^{-1} \ for \ g\in G.$$ 
 Then $\sigma $ is an automorphism of $G$, and $(G,G_{o}, \sigma)$ is a regular homogeneous $k$-manifold. Also, $M\cong G/G_{o}$ and the symmetries $s_{x}$ are given by the formula \ref{eq1} for $s=s_{o}.$
\end{enumerate}
\label{tt25}
\end{theorem}

\noindent
Assume now, that $(M,\omega , s,o)$ is a $k$-symmetric, symplectic pointed space and let $(\mathfrak{g}, \nu )$ be the corresponding $k$-symmetric pair. There is a differential 2-form $\underline{\Omega}$ on $\mathfrak{g}$   given by the pullback of $\omega \mid_{o}$ by $(\pi^{\ast} \omega_{o})_{e},$ where $\pi :G \rightarrow M$ is the canonical projection. The form $\underline{\Omega}$ is closed  and nondegenerate on $\mathfrak{m}.$ The form $\omega$ is invariant under the action of $G\subseteq Aut(M, \omega ,s)$ and every $s_{x}$, so the form $\underline{\Omega}$ is $ad_{\mathfrak{h}}$-invariant and $\nu = \sigma_{\ast}$-invariant.
\\
Conversely - take $(\mathfrak{g}, \nu , \Omega)$ where $\Omega$ is $ad_{\mathfrak{h}}$- and $\nu$-invariant symplectic form on $\mathfrak{m}$. Let $(G,H, \sigma)$ be the corresponding simply connected, $k$-symmetric homogeneous manifold. Let $\underline{\Omega}$ be a 2-form on $\mathfrak{g}$ obtained from $\Omega$ by extending by 0 in $\mathfrak{h}$, so that $i(X)\underline{\Omega} =0$ for $x \in \mathfrak{h}.$ We get a nondegenerate $G$-invariant form $\omega$ on $G/H$ such that $(\underline{\omega} =\pi^{\ast} \omega)$. Let $C^{p} (\mathfrak{g} )$ denote the $p$-cochains in $\mathfrak{g}$ over reals and let $\delta$ be a coboundary operator for the trivial representation of $\mathfrak{g}$ on $\mathbb{R}.$ Because $\underline{\Omega}$ is $ad_{\mathfrak{h}}$-invariant, it is a Chevalley 2-cocycle. This means that $\omega $ is closed. We conclude that, by equation \ref{eq1}, $\omega$ is invariant with respect to symmetries $s_{x}$.
\newline

\begin{definition}
{\rm A triple $(\mathfrak{g} , \nu , \Omega)$ is called a {\sl prime} $k$-symmetric symplectic triple, if the following conditions are satisfied:}
 \begin{itemize}
  \item {\rm $\nu$ is an automorphism of $\mathfrak{g}$ such that $\nu^{k}=id$ and $k\geq 2$ is the least number with this property;}
  \item {\rm for the canonical decomposition $\mathfrak{g} =\mathfrak{h} \oplus \mathfrak{m},$ the Lie algebra $\mathfrak{h}$ does not contain any proper ideal of $\mathfrak{g}$ ($\Leftrightarrow \ \mathfrak{h}$ acts effectively $\mathfrak{m}$ ) and} 
  $$[\mathfrak{m} ,\mathfrak{m}]_{\mathfrak{h}} =\mathfrak{h}.$$
  \item {\rm $\Omega$ is an $ad_{\mathfrak{h}}$-invariant symplectic form on $\mathfrak{m}$ and $\nu \mid_{\mathfrak{m}}$ is a symplectomorphism of the symplectic vector space $(\mathfrak{m},\Omega)$}.
 \end{itemize}
{\rm If  $\mathfrak{g}$ is semisimple (simple, reductive), then the triple $(\mathfrak{g}, \nu)$ is called semisimple (simple, reductive)}.
\end{definition}

\begin{theorem}\label{thm:triples}
For every $k \geq 2$ there is a one-to-one correspondence between prime symplectic $k$-symmetric triples $(\mathfrak{g} , \nu , \Omega)$ and simply connected $k$-symmetric pointed spaces $(M,\omega,s,o)$. This correspondence is described as follows: $\mathfrak{g}$ is the Lie algebra of $G(M)$, $\nu$ is the differential of the automorphism of $G(M)$ generated by $s_{o}$, and $\Omega$ is the pullback of $\omega$ by the canonical projection.
\end{theorem}

\section{Singling out symplectic $k$-symmetric triples}

In this section we will effectively describe how one can single out a $k$-symmetric symplectic pair in the class of  $k$-symmetric  triples. We begin with the following proposition.

\begin{proposition}
If $(\mathfrak{g}, \nu, \Omega)$ is semisimple then the Killing form of $\mathfrak{g}$ is nondegenerate on $\mathfrak{h}$ and on $\mathfrak{m}$.
\end{proposition}

\begin{proof}
If $(\mathfrak{g}, \nu, \Omega)$ is a semisimple $k$-symmetric triple then the Killing form $B$ of $\mathfrak{g}$ is nondegenerate. Moreover:
$$B(x,y)=B(\nu x, \nu y) \ for \ x,y \in \mathfrak{g} $$
Thus for $x\in \mathfrak{h}$ and $y \in \mathfrak{g}$ we have
$$B(x,y)=B(x, \nu y) \ \Rightarrow \ B(x, y-\nu y)=0 \ \Rightarrow \ B(x,Ay)=0$$
But $A(\mathfrak{g})=\mathfrak{m}$ so $\mathfrak{h}$ and $\mathfrak{m}$ are $B$-orthogonal.
\end{proof}

\noindent
The following  fact is well known (one can consult \cite{r}). 

\begin{lemma}\label{lemma:whitehead}
Let $\mathfrak{g}$ be a semisimple finite-dimensional Lie algebra over a field $\mathbb{K}$ of characteristic 0. Let $M$ be a finite-dimensional $\mathfrak{g}$- module. Then:
\begin{enumerate}
\item $H^{1}(\mathfrak{g},M)=0,$
\item $H^{2}(\mathfrak{g},M)=0.$
\end{enumerate}
\end{lemma}

\noindent
It follows that

$$\Omega =-da \ \ a\in \mathfrak{g}^{\ast}.$$

\noindent
For any Lie algebra $\mathfrak{t}$ we have $b\in \mathfrak{t} \Rightarrow db(X,Y)=-b([X,Y])$ so:

$$-i(h)\Omega(X)=0 \Rightarrow a([\mathfrak{h} , \mathfrak{m}]) =a([\mathfrak{h} , \mathfrak{h}]) =0.$$

\noindent
Since the Killing form $B$ is nondegenerate, the 1-form $a$ is dual to some element $Z\in \mathfrak{g}$ with respect to $B$ by the formula

$$B(Z,\bullet )=a(\bullet )$$
$$\forall_{X,Y\in \mathfrak{g}} \ \Omega (X,Y)= B(Z,[X,Y]).$$

\noindent
The latter implies the following result.

\begin{theorem}\label{thm:sympl-descr}
A semisimple $k$-symmetric pair $(\mathfrak{g},\nu)$ is symplectic if and only if there is $Z\in Z(\mathfrak{h})$ such that $ker(ad_{Z}|_{\mathfrak{m}})={0}.$
\end{theorem}

\begin{proof}
Let $(\mathfrak{g},\nu)$ be symplectic. There is $Z\in \mathfrak{g}$ such that:
$$\Omega (X,Y)=B(Z, [X,Y])$$
Let $Z=h+m,$ we have:
$$\forall_{h_{1}\in \mathfrak{h}} \ \forall_{m_{1}\in \mathfrak{m}} \ B(Z,[h_{1},m_{1}])=0$$
$$B([h+m,h_{1}],m_{1})=0$$
but $\mathfrak{h}$ is a subalgebra, also $\mathfrak{m}$ and $\mathfrak{h}$ are $B$ - orthogonal so:
$$\forall_{h_{1}\in \mathfrak{h}} \ \forall_{m_{1}\in \mathfrak{m}} \ B([m,h_{1}],m_{1})=0.$$
The Killing form is nondegenerate on $\mathfrak{m}$ and $[\mathfrak{h},\mathfrak{m}] \subset \mathfrak{m}$ therefore:
$$\forall_{h_{1} \in \mathfrak{h}} \ [h_{1},m]=0.$$
Assume, that $m \neq 0.$ The symplectic form is nondegenerate on $\mathfrak{m}$:
$$\exists_{m_{2} \in \mathfrak{m}} \ B(Z,[m,m_{2}]) \neq 0$$
$$B([h+m,m],m_{2}) \neq 0$$
$$B([h,m]+[m,m],m_{2}) \neq 0$$
$$B(0,m_{2}) \neq 0,$$
a contradiction, which implies $Z=h.$ Because $i(h)\Omega(X)=0$ we get:
$$\forall_{h_{1},h_{2}\in \mathfrak{h}} \ B(h,[h_{1},h_{2}])=0$$
$$B([h,h_{1}],h_{2})=0,$$
and because $B$ is nondegenerate on $\mathfrak{h}$ we obtain:
$$\forall_{h_{1}\in \mathfrak{h}} \ [h,h_{1}]=0,$$
which implies $Z\in Z(\mathfrak{h}).$

Conversely,  define a 2-form $\Omega$ by
$$\Omega (X,Y):=B(Z,[X,Y])=B([Z,X],Y).$$
It is nondegenerate on $\mathfrak{m}$ because $B$ is nondegenerate on $\mathfrak{m}$ and $i(W) \Omega =0$ for $W \in \mathfrak{h}$. Furthermore $\Omega$ is closed by  Lemma \ref{lemma:whitehead}. Now it is sufficient to notice, that $\nu (Z)=Z$. The latter follows from the equalities below.
$$B(Z,[\nu (X), \nu (Y)])=B(Z,\nu [X,Y])=B(\nu (Z),\nu [X,Y])=B(Z,[X,Y]).$$
\end{proof}

\noindent
In view of the latter theorem, it will be convenient to use the following terminology.

\begin{definition}
An element $Z$ such that $Z\in Z(\mathfrak{h})$ and $ker(ad_{Z}|_{\mathfrak{m}})={0}$ will be called injective.
\end{definition}
Now Theorem \ref{thm:sympl-descr} can be fomulated as follows.
{\sl A semisimple $k$-symmetric pair $(\mathfrak{g},\nu)$ is symplectic if and only if it has an injective element}.

\section{Classification of $3$-symmetric symplectic spaces}

\subsection{Classification of 3-symmetric spaces}

 In \cite{g1} Gray and Wolf presented (inter alia) a classification of simple, simply connected 3-symmetric manifolds $X=G/H,$ where group $G$ acts effectively on $X.$ Compact and noncompact spaces are treated separately. In the first case the classification is divided into the following three subcases :
\begin{enumerate}
\item[C1:] $G$ is a simple Lie group, $H$ is a centralizer of a torus, $rank \ G \ = \ rank \ H$
\item[C2:] $G$ is a simple Lie group, $H$ is  semisimple with center $\mathbb{Z}_{3}$, $rank \ G \ = \ rank \ H$
\item[C3:] $rank \ G \ > \ rank \ H$
\end{enumerate}
At the second step of their analysis, they use the above list and  study the noncompact case by a certain construction which we are going to describe now. Here are the basic steps of it.
\\
{\bf \em Step 1.} Let $G$ be a reductive non-compact Lie group with Lie algebra $\mathfrak{g}$ and let $H$ be a closed and reductive subgroup of $G$ with the Lie algerba $\mathfrak{h}.$ Assume that $G$ acts effectively on $X=G/H.$ There exists a Cartan involution $\sigma$ of $\mathfrak{g}$ which preserves $\mathfrak{h}$ and we can decompose $\mathfrak{g}$ into $(+1)1$- and $(-1)$-eigenspaces of $\sigma$ (this is the Cartan decomposition):

$$\mathfrak{g}=\mathfrak{g}^{\sigma}+\mathfrak{m} \ \ \ \ \mathfrak{h}=\mathfrak{h}^{\sigma}+(\mathfrak{h}\cap \mathfrak{m}).$$

We obtain compact real forms of $\mathfrak{g}^{\mathbb{C}}$ and $\mathfrak{h}^{\mathbb{C}}:$

$$\mathfrak{g}_{u}=\mathfrak{g}^{\sigma}+\sqrt{-1}\mathfrak{m} \ \ \ \ \mathfrak{h}_{u}=\mathfrak{h}^{\sigma}+\sqrt{-1}(\mathfrak{h}\cap \mathfrak{m})$$

\noindent
Then the following result holds.

\begin{lemma}[Lemma 7.4 in \cite{g1}]
There is a unique choice of Lie group $G_{u}$ with Lie algebra $\mathfrak{g}_{u}$ which has the properties ($Z_{u}$ denotes the identity component of the center of $G_{u}$):
\begin{enumerate}
\item the analytic subgroup $H_{u}$ for $\mathfrak{h}_{u}$ is closed,
\item the action of $G_{u}$ on $X_{u}=G_{u}/H_{u}$ is effective,
\item $X_{u}^{'}=G_{u}/Z_{u}H_{u}$ is simply connected, the natural projection $X_{u} \rightarrow X_{u}^{'}$ is a principal torus bundle with group $Z_{u},$ and $\pi_{1}(X_{u})\cong \pi_{1}(Z_{u}),$ free abelian of rank $dimZ_{u}.$
\end{enumerate}
\label{aa2}
\end{lemma}

{\bf \em Step 2.} Let $X=G/H$ be a coset space of compact and connected Lie groups $G$ and $H$ with Lie algebras $\mathfrak{g}$ and $\mathfrak{h}.$ Let $G$ act effectively on $X$ and let $\sigma$ be an involutive automorphism of $\mathfrak{g}$ which preserves $\mathfrak{h}.$ In the compact case such automorphism always exists, since a compact Lie group contains a circle subgroup, say $S$. Thus, one can take an inner automorphism generated by $s\in S\subset H$ of order $2$. We have the following decomposition into $(+1)$- and $(-1)$-eigenspaces of $\sigma$:

$$\mathfrak{g}=\mathfrak{g}^{\sigma}+\mathfrak{m} \ \ \ \ \mathfrak{h}=\mathfrak{h}^{\sigma}+(\mathfrak{h}\cap \mathfrak{m}).$$

\noindent
This decomposition defines real forms of $\mathfrak{g}^{\mathbb{C}}$ and $\mathfrak{h}^{\mathbb{C}}:$

\begin{equation}
\mathfrak{g}^{\ast}=\mathfrak{g}^{\sigma}+\sqrt{-1}\mathfrak{m} \ \ \ \ \mathfrak{h}^{\ast}=\mathfrak{h}^{\sigma}+\sqrt{-1}(\mathfrak{h}\cap \mathfrak{m}).
\label{eq3}
\end{equation}

\noindent
Furthermore, $\mathfrak{g}^{\ast}$ is reductive, $\mathfrak{h}^{\ast}$ is reductive in $\mathfrak{g}^{\ast},$ and the following result holds.

\begin{lemma}[Lemma 7.5 in \cite{g1}]
There is a unique simply connected coset space $G^{\ast}/H^{\ast}$, where $G^{\ast}$ is a connected Lie group with the Lie algebra $\mathfrak{g}^{\ast},$ $H^*$ is a closed subgroup of $G$ with the Lie algebra $\mathfrak{h}^{\ast},$ and $G^{\ast}$ acts effectively on $X^{\ast}.$
Let $F$ be a torsion subgroup of $\pi_{1}(X).$ Then $F$ can be viewed as a finite central subgroup of $G^{\ast}_{u}$ (cf. Step 1) such that $G=G^{\ast}_{u}/F,$ $H=(H^{\ast}_{u}F)/F$ and $X=X^{\ast}_{u}/F.$
\label{aa1}
\end{lemma}

{\bf \em Step 3.} Let $(\mathfrak{g}^{\ast}, \nu)$ be a noncompact prime 3-symmetric pair and assume that $\mathfrak{g}^{\ast}$ is semisimple. We will show, that it can be obtained from some compact $3$-symmetric space by the procedure from Lemma \ref{aa2}.
This pair defines a $3$-symmetric space $X^{\ast}=G^{\ast}/H^{\ast}.$ First, we will assign a compact $3$-symmetric space $X$ to $X^{\ast}$.
Extend $\nu$ by linearity to an automorphism of $(\mathfrak{g}^{\ast})^{\mathbb{C}}$ and let $L$ be a maximal compact subgroup of $Int((\mathfrak{g}^{\ast})^{\mathbb{C}})$ containing $\nu.$ It is known that for any complex semisimple Lie algebra $\mathfrak{a}$, any maximal compact subgroup of a Lie group $A=Int(\frak{a})$ corresponds to a compact real form of $\frak{a}$. Because of that, we may conclude that $L$ defines a compact real form $\mathfrak{g}$ of $\mathfrak{g}^{\ast \mathbb{C}}$ by  $exp(ad\mathfrak{g})$ which is the identity component $L_{0}$ of $L.$ We have $\nu(\mathfrak{g})=\mathfrak{g},$ so let $\mathfrak{h} = \mathfrak{g}^{\nu}$ and obtain a simply connected homogeneous space $X=G/H$  with connected $G$  acting effectively on $X$. 
Using lemma \ref{aa2} we also have a correspondence between 
$X^{\ast}=G^{\ast}/H^{\ast}$ and  $ X^{\ast}_{u}=G^{\ast}_{u}/K^{\ast}_{u}.$
As $\mathfrak{g}$ and $\mathfrak{g}^{\ast}_{u}$ are compact real forms, there is an automorphism $\alpha$ of $(\mathfrak{g}^{\ast})^{\mathbb{C}}$ sending $\mathfrak{g}$ to $\mathfrak{g}_{u}^{\ast}.$ Furthermore $\mathfrak{h}$ and $\mathfrak{h}_{u}^{\ast}$ are compact real forms of $(\mathfrak{g}^{\ast})^{\mathbb{C}})^{\nu}$ so we have $\alpha(\mathfrak{h})\cong \mathfrak{h} \cong \mathfrak{h}_{u}^{\ast},$ thus there is also an automorphism $\beta$ of $\alpha(\mathfrak{g}) =\mathfrak{g}^{\ast}_{u}$ which sends $\alpha (\mathfrak{h})$ to $\mathfrak{h}_{u}^{\ast}.$ Therefore $\beta \alpha : \mathfrak{g} \cong  \mathfrak{g}_{u}^{\ast}$ sends $\mathfrak{h}$ to $\mathfrak{h}_{u}^{\ast},$ and so induces an isomorphism from $X$ to $X_{u}^{\ast}.$ As a result we can view $X^{\ast}=G^{\ast}/H^{\ast}$ as a space constructed from $X=G/H$ by the procedure described in Lemma \ref{aa1}, provided that we view $\mathfrak{h}:=\mathfrak{g}^{\phi},$ where an automorphism of order $3$ is defined as $\phi =\beta \circ \alpha \circ \nu \circ \alpha^{-1} \circ \beta^{-1}.$ 
We can summarize these observations as follows. We will say that a $3$-symmetric space of non-compact type $G^*/H^*$ is obtained from a $3$-symmetric space of compact type {\it by a canonical procedure}, if $G^*/H^*$ corresponds to $G/H$ in accordance with Lemma \ref{aa1}.  

\begin{theorem}\label{thm:correspondence} Any simple $3$-symmetric  space $G^*/H^*$ of non-compact type is obtained from some compact $3$-symmetric simple space by a canonical procedure.
\end{theorem}

Now, one only needs to check what spaces can be obtained from { [C1]-[C3]} by the procedure from Lemma \ref{aa1}, and therefore obtain the classification of noncompact simple 3-symmetric spaces $X^{\ast}=G^{\ast}/H^{\ast}$ where $G^{\ast}$ acts effectively:

\begin{enumerate}
\item[NC1]: $G^{\ast}$ is a simple Lie group, $H^{\ast}$ is  a centralizer of a compact toral subgroup. This class consists of $G^{\ast},$ which are obtained from spaces in {Z1} (by the procedure from Lemma \ref{aa1} or by complexification).
\item[NC2]: $G^{\ast}$ is a simple Lie group, $H^{\ast}$ is not a centralizer of a compact toral subgroup, $rank \ G^{\ast} \ = \ rank \ H^{\ast}$. This class consists of $G^{\ast},$ which are obtained from spaces in {C2} (by the procedure from Lemma \ref{aa1} or by complexification)
\item[NC3]: $rank \ G^{\ast} \ > \ rank \ H^{\ast}$
\end{enumerate}

\subsection{Classification of $3$-symmetric symplectic spaces}

We begin with the compact case.

\begin{theorem}[Theorem 9.5 in \cite{g1}]
Let $X=G/H$ where $G$ is a compact Lie group,  $H$ is a centralizer of a torus. Then $X$ is equipped with an invariant K\"ahler structure.
\label{95}
\end{theorem}

\noindent
Therefore, all spaces from the class {C1} are symplectic. Classes {C2} and {C3} have no symplectic representatives, because in each simple space the center of the isotropy subgroup is discrete and, therefore, has no injective elements. All spaces from {C1} are presented in table \ref{tab1}.

Now we can analyze the noncompact case. Let $\sigma$ be an involutive automorphism of $\mathfrak{g}$. Let
 $$\mathfrak{g}=\mathfrak{h} \oplus \mathfrak{m}_{A}$$ 
be  the canonical reductive decomposition, determined by some automorphism of order $3$. Note that this time we denote the reductive complement by $\mathfrak{m}_A$, since $\frak{m}$ is reserved for the $(-1)$-eigenspace of the involution $\sigma$.  We can assume that  $\sigma (\mathfrak{h})= \mathfrak{h}.$ Then

$$B(\sigma (\mathfrak{h}),\sigma (\mathfrak{m}_{A}))=B(\mathfrak{h},\mathfrak{m}_{A})=0$$
hence
$$B(\mathfrak{h}, \sigma (\mathfrak{m}_{A}))=0.$$
Because the Killing form $B$ is nondegenerate on $\mathfrak{h},$ we get the equality $\sigma (\mathfrak{m}_{A})=\mathfrak{m}_{A}.$ Therefore the following equalities yield a decomposition into eigenspaces of $\sigma :$
$$\mathfrak{h}= \mathfrak{h}_{1} \oplus \mathfrak{h}_{2} \ \ \mathfrak{m}_{A}= \mathfrak{m}_{1} \oplus \mathfrak{m}_{2}$$
$$\mathfrak{h}^{\ast}= \mathfrak{h}_{1} \oplus i\mathfrak{h}_{2}$$

\noindent
We need the following result.

\begin{lemma}\label{cen1}
Assume that $G/H$ is $3$-symmetric and compact. If $dimZ(\mathfrak{h})=1,$
then any $3$-symmetric homogeneous space $G^*/H^*$ obtained by a canonical procedure, also possesses an injective element (and, therefore, is symplectic).
\end{lemma}

\begin{proof}
Every element $g\in \mathfrak{g}^{\ast}$ and $h\in \mathfrak{h}^{\ast}$ can be written as
$$g=h^{1}+m^{1}_{A}+ih^{2}+im^{2}_{A} \ \ \ h=h^{'}+ih^{''},$$
for $h^{1}, \ h^{2}, \ h^{'}, \ h^{''} \in \mathfrak{h}$ and $m^{1}_{A}, \ m^{2}_{A} \in \mathfrak{m}_{A}.$ Moreover, an automorphism $\sigma$ acts as the identity on $h^{1}, \ h^{'}, \ m^{1}_{A}$ and minus the identity on $h^{2}, \ h^{''}, \ m^{2}_{A}.$
Let $Z=h^{'}+h^{''}$ be an injective element in the center of $\mathfrak{h},$ and let $h=h_{1}+h_{2}$ be any element in $\mathfrak{h}$ decomposed with respect to $\sigma $. The following equalities hold
$$[h^{'}+h^{''},h_{1}]=0$$
$$[h^{'},h_{1}]+[h^{''},h_{1}]=0$$
$$\sigma ([h^{'},h_{1}]+[h^{''},h_{1}])= \sigma (0)$$
$$[h^{'},h_{1}]-[h^{''},h_{1}]=0$$
We have $[h^{'},h_{1}]=[h^{''},h_{1}]=0,$ analogously
$$[h^{'}+h^{''},h_{2}]=0$$
gives us $[h^{'},h_{2}]=[h^{''},h_{2}]=0.$ Therefore, $h^{'},h^{''} \in Z(\mathfrak{h}).$ The center of $\mathfrak{h}$ is one-dimensional, and $\sigma$ acts on $h^{'}$ by $id,$ and by $-id$ on $h^{''},$ so $h^{'}=0$ or $h^{''}=0.$ Without loss of generality assume $h^{''}=0,$ therefore $h^{'}$ is injective. Let:
$$[h^{'},h^{1}+m^{1}_{A}+ih^{2}+im^{2}_{A}]=0$$
for any $g=h^{1}+m^{1}_{A}+ih^{2}+im^{2}_{A}$ in $\mathfrak{g}^{\ast}.$ Vector $Z=h^{'}$ is in the center of $\mathfrak{h}$, so
$$[h^{'},h^{1}+m^{1}_{A}+ih^{2}+im^{2}_{A}]=[h^{'},m^{1}_{A}+im^{2}_{A}]=[h^{'},m^{1}_{A}]+i[h^{'},m^{2}_{A}]=0$$
therefore $m^{1}_{A}=m^{2}_{A}=0$, because $Z$ is injective on $\mathfrak{m}_{A}.$
\end{proof}

\noindent
We also need the following properties.

\begin{theorem}[Theorem 13.3(1) in \cite{mm}]
Let $X=G/H,$ where $H$ is a connected subgroup of maximal rank in a compact connected Lie group $G.$ Let $\sigma$ be an automorphism of $G$ which preserves $H,$ thus acts on $X,$ and preserves some $G$-invariant almost complex structure on $X.$ Then the following conditions are equivalent, and each implies that $\sigma$ preserves every $G$-invariant almost complex structure on $X:$
\begin{enumerate}
\item $\sigma$ is an inner automorphism of $G,$
\item $\sigma =ad(k)$ for some element $k\in H.$ 
\end{enumerate} 
\end{theorem}

\noindent
We also have theorem 7.7(ii) from \cite{g1} for a special case $\Sigma= \{ 1,\phi, \phi^{2} \}$ where $\phi$ is an automorphism of $\mathfrak{g}^{\ast},$ which defines a 3-symmetric structure.

\begin{theorem}
Let $X=G/H$ where $G$ is compact connected Lie group with closed and connected subgroup $H$ and let $G$ act effectively on $X.$ Let $\sigma$ be an involutive automorphism of $G$ which preserves $H.$ Let $X^{\ast}=G^{\ast}/K^{\ast}$ be the simply connected space described in lemma \ref{aa1}. Then:
\begin{enumerate}
\item $G$-invariant and $\sigma$-invariant almost complex structures on $X$ are in one-to-one correspondence with $G^{\ast}$-invariant and $\sigma$-invariant almost complex structures on $X^{\ast}.$
\item $X^{\ast}$ has a $G^{\ast}$-invariant and $\sigma$-invariant almost complex structure.
\end{enumerate}
\end{theorem}

Now we can continue our classification of symplectic 3-symmetric spaces. We will show, that all spaces in {NC1} are symplectic. To do this, we have to prove, that every space in this class has an injective element. All such spaces are constructed from $G/H$ from {C1}, and so $G/H$ has an invariant K\"ahler structure $((\cdot,\cdot),\omega,J).$  
Note that in what follows we will use the well-known relation between the  Riemannian metric $(\cdot,\cdot)$, the complex structure $J$ and the K\"ahler (symplectic) form $\omega$:

$$(\cdot,\cdot)=\omega(\cdot,J\cdot).$$
Also, we will consider invariant almost complex structures on reductive homogeneous spaces $G/H$. A general description of such structures can be found in \cite{kn} (Proposition 6.5, Chapter X). We use the following: there is a one-to-one correspondence between $G$-invariant almost complex structures on $G/H$ and linear endomorphisms $J: \frak{m}\rightarrow \frak{m}$such that $J^2=-id$ and $J\circ Ad\,h=Ad\,h\circ J$ for any $h\in H$. If $H$ is connected, the latter is equivalent to the identity
$$J\circ ad\,t=ad\,t\circ J, \forall t\in\frak{h}.$$

 We will consider two cases determined by the type of involutive automorphism $\sigma$ (it may be inner or outer). We should keep in mind that the spaces belonging to the class NC1 are obtained by the canonical procedure from C1 and correspond either to inner $\sigma$, or to outer $\sigma$. For example, all non-compact 3-symmetric spaces $G^*/H^*$ from Table 3 {\it belong to} NC1. 
\newline
{\bf \em Case 1}: $\sigma$ is an inner automorphism. Then, based on above properties, $\sigma = ad(k)$ for some $k\in H$ and $\sigma$ preserves any $G$-invariant almost complex structure. Assume that the injective element for the form $\omega$ is $Z=h^{'}+h^{''}$ (the decomposition is taken with respect to $\sigma)$). As in the proof of Lemma \ref{cen1} we show that $h^{'},h^{''} \in Z(\mathfrak{h}).$ Then
$$\forall_{X,Y\in \mathfrak{m}_A} \ \omega(X,Y)=B(h^{'}+h^{''},[X,Y]).$$
Assume that $h^{'}\neq 0$ and $M_{1}+M_{2}\neq 0$ is in the kernel of $ad_{h'}$. Then
$$[h^{'},M_{1}+M_{2}]=0$$
$$\sigma([h^{'},M_{1}+M_{2}])=\sigma(0)$$
$$[h^{'},M_{1}-M_{2}]=0$$
so $[h^{'},M_{1}]=[h^{'},M_{2}]=0.$ Without loss of generality take $M_{1}\neq 0.$ Because $J$ is $G$ - invariant we have $0=J[h^{'},M_{1}]=[h^{'},JM_{1}].$ Therefore:
$$0 < (M_{1},M_{1}) = \omega(JM_{1},M_{1})=B(h^{'}+h^{''},[JM_{1},M_{1}])=B(h^{''},[JM_{1},M_{1}])$$
Because $J$ is $\sigma$-invariant and $\sigma(M_{1})=M_{1}$ we have:
$$B(h^{''},[JM_{1},M_{1}])=B(\sigma (h^{''}),\sigma([JM_{1},M_{1}]))=-B(h^{''},[JM_{1},M_{1}]),$$

$$0<B(h^{''},[JM_{1},M_{1}])=0,$$
a contradiction. Therefore either $h^{'}=0$ or $h^{'}$ has trivial kernel. As in lemma \ref{cen1}  $h^{'}$ or $h^{''}$ is injective and $\mathfrak{h}^{\ast}$ has an injective element.
\\
{\bf \em Case 2}: $\sigma$ is outer. From \cite{g1} we know that all spaces in {NC1} which correspond to an outer automorphism are presented in Table 3. The first two manifolds are constructed from $G=SU(n)/\mathbb{Z}_{n}, \ H=S\{ U(a)\times U(a)/\mathbb{Z}_n  \}$ for $1\leq a\leq n-1.$ Here the center of the isotropy subgroup is one-dimensional, so the existence of an injective element follows from Lemma \ref{cen1}. The last space comes from $G=SO(2n)/\mathbb{Z}_{2}, \ H=\{ U(a)\times SO(2n-2a)  \}/\mathbb{Z}_{2}$ for $1\leq a\leq n-1.$ When $a\neq n-1$ we can again use Lemma \ref{cen1}.
\\
For $a=n-1$ we have $G=SO(2n)/\mathbb{Z}_{2}, \ H=\{ U(n-1)\times SO(2)  \}/\mathbb{Z}_{2}$ and $\sigma = ad(k)$ where $k=diag(k_{1},k_{2}), \ k_{1}\in U(n-1), \ k_{2}\in O(2)$ and $detk_{2}=-1.$ But $G=SO(2n)/\mathbb{Z}_{2}, \ H=\{ U(n-1)\times SO(2)  \}/\mathbb{Z}_{2}$ and $G=SO(2n)/\mathbb{Z}_{2}, \ H=\{ U(n-1)\times O(2)  \}/\mathbb{Z}_{2}$ share the same 3-symmetric pair $(\mathfrak{g}, \nu).$ We have exactly the same canonical decomposition $\mathfrak{g} =\mathfrak{h} + \mathfrak{m}_A$.  Therefore, on the level of the Lie algebras we can treat $\sigma$ as an inner automorphism, and thus every $G$-invariant almost complex structure on $\mathfrak{m}_A$ is $\sigma$-invariant. We can continue as in Case 1.

We should also notice that a direct calculation shows that the following lemma is valid.
\begin{lemma}
The complexification of a Lie algebra   with an injective element has an injective element.
\end{lemma}

\noindent
Classes {NC2} and {NC3} do not contain any symplectic spaces for the same reason as {C2} and {C3}. All spaces from {NC1} are presented in Tables 2 and 3.
We can summarize our considerations as follows.

\begin{theorem}\label{thm:classification}
All simple, simply connected regular 3-symmetric symplectic spaces are given in Table 1 (compact case) and Tables 2 (noncompact case).
\end{theorem}

\begin{rem} {\rm Note that  Table 2 contains Table 3.}
\end{rem}

\noindent

\newpage
\centerline{\bf Tables}

We use standard notation. However, the sources like \cite{he} and \cite{ov} slightly differ from each other, so one should consult \cite{g1} for the details.

\begin{center}
 \begin{table}[h]
 \centering
 {\footnotesize
 \begin{tabular}{| c | c |}
   \hline
   \multicolumn{2}{|c|}{Table 1. Compact spaces} \\
   \hline                        
   $\textbf{G}$ & $\textbf{K}$ \\
   \hline
   $SU(n)/\mathbb{Z}_{n},$  & $S\{ U(a)\times U(b) \times U(c)  \}/\mathbb{Z}_{n},$  \\
   \footnotesize $n\geq2$ & \footnotesize $0\leq a\leq b\leq c, \ 0<b, \ a+b+c=n$ \\
   \hline
   $SO(2n+1)$ & $U(a) \times SO(2n-2a+1),$ \\
   \footnotesize $n\geq 1$ & \footnotesize $1\leq a\leq n$ \\
   \hline
   $Sp(n)/\mathbb{Z}_{2}$ & $\{ U(a) \times Sp(n-a)  \}/\mathbb{Z}_{2},$ \\
    & \footnotesize $1\leq a \leq n$ \\
   \hline
   $SO(2n)/\mathbb{Z}_{2},$ & $\{ U(a)\times SO(2n-2a)  \}/\mathbb{Z}_{2}$ \\
   \footnotesize $n\geq 3$ & $1\leq a \leq n$ \\
   \hline
   $G_{2}$ & $U(2)$ \\
   \hline
   $F_{4}$ & $\{ Spin(7)\times T^{1} \}/\mathbb{Z}_{2}, \ \ \{ Sp(3)\times T^{1}  \}/\mathbb{Z}_{2}$ \\
   \hline
   $E_{6}/\mathbb{Z}_{3}$ & $\{ SO(10)\times SO(2)  \}/\mathbb{Z}_{2}, \ \ \{ S(U(5)\times U(1))\times SU(2)  \}/\mathbb{Z}_{2},$ \\
    & $\{ [SU(6)/\mathbb{Z}_{3}]\times T^{1} \}/\mathbb{Z}_{2}, \ \ \{ [SO(8)\times SO(2)]\times SO(2) \}/\mathbb{Z}_{2}$ \\
   \hline 
   $E_{7}/\mathbb{Z}_{2}$ & $\{ E_{6}\times T^{1} \}/\mathbb{Z}_{3}, \ \ \{ SU(2)\times [SO(10)\times SO(2)] \}/\mathbb{Z}_{2},$ \\ 
    & $\{ SO(2)\times SO(12) \}/\mathbb{Z}_{2}, \ \ S\{ U(7)\times U(1) \}/\mathbb{Z}_{4}$ \\ 
   \hline 
   $E_{8}$ & $SO(14)\times SO(2), \ \ \{ E_{7}\times T^{1} \}/\mathbb{Z}_{2}$ \\
   \hline  
 \end{tabular}
 }
 \caption{$A/\mathbb{Z}_{n}$ denotes the quotient of $A$ by the central subgroup which is cyclic of order n. If $A=B\times C,$ then $A/\mathbb{Z}_{n}$ is $B\times C$ glued along central cyclic subgroups of order n. $S\{...\}$ denotes elements with determinant equal to 1. $SO^{\ast}(n)$ denotes the real form of $SO(n,\mathbb{C}), \ n=2m$ with maximal compact subgroup $U(m).$ }
 \label{tab1}
 \end{table}
\end{center}

\begin{center}
 \begin{table}
 \centering
 {\footnotesize
 \begin{tabular}{| c | c |}
   \hline
   \multicolumn{2}{|c|}{Table 2. Noncompact spaces} \\
   \hline  
   $G^{\ast}$ & $K^{\ast}$  \\
   \hline \hline 
   $SU(n-m,m)/\mathbb{Z}_{n}$ & $S\{ U(a-s,s)\times U(b-t,t)\times U(c-p,p) \}/\mathbb{Z}_{n}$ \\
     & \footnotesize $a+b+c=n, \ s+t+p=m, \ 0\leq a\leq b \leq c, \ 1\leq b,$    \\
     &  \footnotesize $0\leq 2s \leq a, \ 0\leq 2t \leq b, \ 0\leq 2p \leq c$ \\
   \hline
   $SL(n,\mathbb{R})/\mathbb{Z}_{2}$ & $\{ SL(\frac{n}{2},\mathbb{C})\times T^{1} \}/\mathbb{Z}_{\frac{n}{2}}$ \\
     & \footnotesize $n\equiv 0(2)$ \\
   \hline
   $SL(\frac{n}{2},\mathbb{H})/\mathbb{Z}_{2}$ & $\{ SL(\frac{n}{2},\mathbb{C})\times T^{1} \}/\mathbb{Z}_{\frac{n}{2}}$ \\
     & \footnotesize $n\equiv 0(2)$ \\  
   \hline
   $SL(n,\mathbb{C})/\mathbb{Z}_{n}$ & $S \{ GL(a,\mathbb{C})\times GL(b,\mathbb{C})\times GL(c,\mathbb{C}) \}/\mathbb{Z}_{n}$ \\
     & \footnotesize $a+b+c=n, \  0\leq a\leq b \leq c, \ 1\leq b$ \\
   \hline \hline
   $SO(2n+1-2s-2t,2s+2t)$ & $U(a-s,s)\times SO(2n-2a+1-2t,2t)$ \\
     & \footnotesize  $1\leq a\leq n, \ 0\leq 2s\leq a$ \\
   \hline
   $SO(2n+1, \mathbb{C})$  &  $GL(a,\mathbb{C})\times SO(2n-2a+1, \mathbb{C})$  \\
     & \footnotesize  $1\leq a\leq n$  \\
   \hline \hline
   $Sp(n-s-t,s+t)/\mathbb{Z}_{2}$ & $\{ U(a-s,s)\times Sp(n-a-t,t) \}/\mathbb{Z}_{2}$ \\
     & \footnotesize $1\leq a \leq n, \ 0\leq 2s \leq a, \ 0\leq 2t \leq n-a$ \\
   \hline
   $Sp(n,\mathbb{R})/\mathbb{Z}_{2}$ & $\{ U(a-s,s)\times Sp(n-a, \mathbb{R}) \}/\mathbb{Z}_{2}$ \\
     & \footnotesize $1\leq a \leq n, \ 0\leq 2s \leq a$ \\
   \hline
   $Sp(n,\mathbb{C})/\mathbb{Z}_{2}$ & $\{ GL(a,\mathbb{C})\times Sp(n-a,\mathbb{C}) \}/\mathbb{Z}_{2}$ \\
     & \footnotesize $1\leq a\leq n$ \\
   \hline \hline
   $SO(2n-2s-t,2s+t)/\mathbb{Z}_{2}$ & $\{ U(a-s,s)\times SO(2n-2a-t,t) \}\mathbb{Z}_{2}$ \\
     & \footnotesize $1\leq a \leq n, \ 0\leq 2s \leq a, \ 0\leq t \leq n-a$ \\
   \hline
   $SO^{\ast}(2n)/\mathbb{Z}_{2}$ & $\{ U(a-s,s)\times SO^{\ast}(2n-2a) \}\mathbb{Z}_{2}$ \\
     & \footnotesize $1\leq a \leq n, \ 0\leq 2s \leq a$ \\
   \hline
   $SO(2n,\mathbb{C})/\mathbb{Z}_{2}$ & $\{ GL(a,\mathbb{C})\times SO(2n-2a,\mathbb{C}) \}/\mathbb{Z}_{2}$ \\
     & \footnotesize $1\leq a\leq n$ \\
   \hline \hline
   $G_{2}$ & $U(2)$ \\
   \hline
   $G^{\ast}_{2}=G_{2,A_{1}A_{1}}$ & $U(2), \ U(1,1)$ \\
   \hline
   $G^{\mathbb{C}}_{2}$ & $GL(2,\mathbb{C})$ \\
   \hline \hline
   $F_{4}$ & $\{ Spin(7)\times T^{1} \}/\mathbb{Z}_{2}, \ \{ Sp(3)\times T^{1} \}/\mathbb{Z}_{2}$ \\
   \hline
   $F_{4,B_{4}}$ & $\{ Spin(7-r,r)\times T^{1} \}/\mathbb{Z}_{2}, \ \{ Sp(2,1)\times T^{1} \}/\mathbb{Z}_{2}$ \\
     & \footnotesize $ r=0,1$ \\
   \hline  
   $F_{4,C_{3}C_{1}}$ & $\{ Spin(7-r,r)\times T^{1} \}/\mathbb{Z}_{2}, \ \{ Sp(3-t,t)\times T^{1} \}/\mathbb{Z}_{2}$ \\
     & $\{ Sp(3,\mathbb{R})\times T^{1} \}/\mathbb{Z}_{2}$ \\
     & \footnotesize $r=2,3; \ \ t=0,1$ \\
   \hline  
   $F^{\mathbb{C}}_{4}$ & $\{ Spin(7,\mathbb{C})\times C^{\ast} \}/\mathbb{Z}_{2}, \ \{ Sp(3,\mathbb{C})\times\mathbb{C}^{\ast} \}/\mathbb{Z}_{2}$ \\
   \hline \hline
 \end{tabular}
 }
 \label{tab2}
 \end{table}
\end{center}

\begin{center}
 \begin{table}
 \centering
 {\footnotesize
 \begin{tabular}{| c | c |}
   \hline
   \multicolumn{2}{|c|}{Table 2. Noncompact spaces - cont.} \\
   \hline \hline 
   $E_{6}/\mathbb{Z}_{3}$ & $\{ SO(10)\times SO(2) \}/\mathbb{Z}_{2}$ \\
     &  $\ \{ S(U(5)\times U(1))\times SU(2) \}/\mathbb{Z}_{2}, \ \{ [SU(6)/\mathbb{Z}_{3}]\times T^{1} \}/\mathbb{Z}_{2}$ \\
     & $\{ [SO(8)\times SO(2)]\times SO(2) \}/\mathbb{Z}_{2}$ \\
   \hline
   $E_{6,A_{1}A_{5}}$ & $\{ SO^{\ast}(10)\times SO(2) \}/\mathbb{Z}_{2}, \ \{ SO(6,4)\times SO(2) \}/\mathbb{Z}_{2}$ \\
     & $\{ S(U(5-p,p)\times U(1))\times SU(2-s,s) \}/\mathbb{Z}_{2},$ \\
     & $(s,p)=(0,0),(0,1),(0,2),(1,2)$ \\
     & $\{ [SU(6-p,p)/\mathbb{Z}_{3}]\times T^{1}  \}/\mathbb{Z}_{2}, \ p=0,2,3$ \\
     & $\{ [SO^{\ast}(8)\times SO(2)]\times SO(2) \}/\mathbb{Z}_{2}$ \\
     & $\{ [SO(8-p,p)\times SO(2)]\times SO(2) \}/\mathbb{Z}_{2}, \ p=2,4$ \\
   \hline
   $E_{6,D_{5}T^{1}}$ & $\{ SO(10-p,p)\times SO(2) \}/\mathbb{Z}_{2}, \ \{ SO^{\ast}(10)\times SO(2) \}/\mathbb{Z}_{2}, \ p=0,2$ \\
     & $\{ S(U(5-p,p)\times U(1))\times SU(2-s,s) \}/\mathbb{Z}_{2},$ \\
     & $(s,p)=(1,0),(0,1),(1,1),(0,2)$ \\
     & $\{ [SU(6-p,p)/\mathbb{Z}_{3}]\times T^{1} \}/\mathbb{Z}_{2}, \ p=1,2$ \\
     & $\{ [SO^{\ast}(8)\times SO(2)]\times SO(2) \}/\mathbb{Z}_{2}$ \\
     & $\{ [SO(8-p,p)\times SO(2)]\times SO(2) \}/\mathbb{Z}_{2}, \ p=0,2$ \\
   \hline
   $E^{\mathbb{C}}_{6}/\mathbb{Z}_{3}$ & $\{ SO(10,\mathbb{C})\times \mathbb{C}^{\ast} \}/\mathbb{Z}_{2}$ \\
     & $\{ S(GL(5,\mathbb{C})\times\mathbb{C}^{\ast})\times SL(2,\mathbb{C}) \}/\mathbb{Z}_{2}, \ \{ [SL(6,\mathbb{C})/\mathbb{Z}_{3}]\times \mathbb{C}^{\ast} \}/\mathbb{Z}_{2}$ \\
     & $\{ [SO(8,\mathbb{C})\times \mathbb{C}^{\ast}]\times \mathbb{C}^{\ast} \}/\mathbb{Z}_{2}$ \\
   \hline \hline
   $E_{7}/\mathbb{Z}_{2}$ & $\{ E_{6}\times T^{1} \}/\mathbb{Z}_{3}, \ \{ SU(2)\times [SO(10)\times SO(2)] \}/\mathbb{Z}_{2}$ \\
     & $\{ SO(2)\times SO(12) \}/\mathbb{Z}_{2}, \  S(U(7)\times U(1))/\mathbb{Z}_{4}$ \\
   \hline
   $E_{7,A_{7}}$ & $\{ E_{6,A_{1}A_{5}}\times T^{1} \}/\mathbb{Z}_{2}, \ \{ SU(2)\times [SO^{\ast}(10)\times SO(2)] \}/\mathbb{Z}_{2}$ \\
     & $\{ SU(1,1)\times [SO(6,4)\times SO(2)] \}/\mathbb{Z}_{2}$ \\
     & $\{ SO(2)\times SO^{\ast}(12) \}/\mathbb{Z}_{2}, \  \{ SO(2)\times SO(6,6) \}/\mathbb{Z}_{2}$ \\
     & $S(U(7-p,p)\times U(1))/\mathbb{Z}_{4}, \ p=0,3$ \\
   \hline
   $E_{7,A_{1}D_{6}}$ & $\{ E_{6,D_{5}T^{1}}\times T^{1} \}/\mathbb{Z}_{2}, \ \{ E_{6,A_{1}A_{5}}\times T^{1} \}/\mathbb{Z}_{2}$ \\
     & $\{ SU(2-p,p)\times [SO(10-s,s)\times SO(2)] \}/\mathbb{Z}_{2},$ \\
     & $(p,s)=(0,0),(0,2),(1,2),(0,4)$ \\
     & $\{ SU(1,1)\times [SO^{\ast}(10)\times SO(2)] \}/\mathbb{Z}_{2}$ \\
     & $\{ SO(2)\times SO(12-p,p) \}/\mathbb{Z}_{2}, \  S(U(7-s,s)\times U(1))/\mathbb{Z}_{4},$ \\
     & $p=0,4 \ s=1,2,3$ \\
   \hline
   $E_{7,E_{6}T^{1}}$ & $\{ E_{6}\times T^{1} \}/\mathbb{Z}_{3}, \ \{ E_{6,D_{5}T^{1}}\times T^{1} \}/\mathbb{Z}_{2}$ \\
    & $\{ SU(1,1)\times [SO(10)\times SO(2)] \}/\mathbb{Z}_{2},  \ \{ SU(2)\times [SO^{\ast}(10)\times SO(2)] \}/\mathbb{Z}_{2}$ \\
    & $\{ SO(2)\times SO^{\ast}(12) \}/\mathbb{Z}_{2}, \ \{ SO(2)\times SO(10,2) \}/\mathbb{Z}_{2}$ \\
    & $S(U(7-s,s)\times U(1))/\mathbb{Z}_{4}, \ s=1,2$ \\
   \hline
   $E_{7}^{\mathbb{C}}/\mathbb{Z}_{2}$ & $\{ E_{6}^{\mathbb{C}}\times \mathbb{C}^{\ast} \}/\mathbb{Z}_{3}, \ \{ SL(2,\mathbb{C})\times [SO(10,\mathbb{C})\times \mathbb{C}^{\ast}] \}/\mathbb{Z}_{2}$ \\
    & $\{ \mathbb{C}^{\ast}\times SO(12,\mathbb{C}) \}/\mathbb{Z}_{2}, \ S\{ GL(7,\mathbb{C})\times \mathbb{C}^{\ast} \}/\mathbb{Z}_{4}$ \\
   \hline \hline
   $E_{8}$ & $SO(14)\times SO(2), \ \{ E_{7}\times T^{1} \}/\mathbb{Z}_{2}$ \\
   \hline
   $E_{8,D_{8}}$ & $SO(14)\times SO(2), \ SO(8,6)\times SO(2), \ SO^{\ast}(14)\times SO(2)$ \\
     & $\{ E_{7,A_{1}D_{6}}\times T^{1} \}/\mathbb{Z}_{2}, \ \{ E_{7,A_{7}}\times T^{1} \}/\mathbb{Z}_{2}$ \\
   \hline
   $E_{8,A_{1}E_{7}}$ & $SO(12,2)\times SO(2), \ SO(10,4)\times SO(2), \ SO^{\ast}(14)\times SO(2)$ \\
     & $\{ E_{7}\times T^{1} \}/\mathbb{Z}_{2}, \ \{ E_{7,E_{6}T^{1}}\times T^{1} \}/\mathbb{Z}_{2}, \ \{ E_{7,A_{1}D_{6}}\times T^{1} \}/\mathbb{Z}_{2}$ \\
   \hline
   $E_{8}^{\mathbb{C}}$ & $SO(14,\mathbb{C})\times \mathbb{C}^{\ast}, \ \{ E_{7}^{\mathbb{C}}\times \mathbb{C}^{\ast} \}/\mathbb{Z}_{2}$ \\
   \hline \hline
 \end{tabular}
 }
 \label{tab23}
 \end{table}
\end{center}

\begin{center}
 \begin{table}
 \centering
 {\footnotesize
 \begin{tabular}{| c | c |}
   \hline
   \multicolumn{2}{|c|}{Table 3. Noncompact spaces of outer automorphism} \\
   \hline
   $SL(n,\mathbb{R})/\mathbb{Z}_{2}$ & $\{ SL(\frac{n}{2},\mathbb{C})\times T^{1} \}/\mathbb{Z}_{\frac{n}{2}}$ \\
     & \footnotesize $n\equiv 0(2)$ \\
   \hline
   $SL(\frac{n}{2},\mathbb{H})/\mathbb{Z}_{2}$ & $\{ SL(\frac{n}{2},\mathbb{C})\times T^{1} \}/\mathbb{Z}_{\frac{n}{2}}$ \\
     & \footnotesize $n\equiv 0(2)$ \\
   \hline
   $SO(2n-2s-t,2s+t)/\mathbb{Z}_{2}$ & $\{ U(a-s,s)\times SO(2n-2a-t,t) \}\mathbb{Z}_{2}$ \\
     & \footnotesize $1\leq a \leq n, \ 0\leq 2s \leq a, \ 0\leq t \leq n-a$ \\
   \hline
 \end{tabular}
 }
 \label{tab3}
 \end{table}
\end{center}

Department of Mathematics and Computer Science

University of Warmia and Mazury

S\l\/oneczna 54, 10-710, Olsztyn, Poland

e-mail: hoh@poczta.onet.pl (MB), 

tralle@matman.uwm.edu.pl (AT)

\end{document}